\documentclass[12pt]{amsproc}
\usepackage[top=1in, bottom=1in, left=.7in, right=.7in]{geometry}
\usepackage{color}
\usepackage{amssymb}
\usepackage{amsthm}
\usepackage{amsmath}
\usepackage[colorlinks=true,urlcolor=blue,linkcolor=blue,citecolor=black]{hyperref}
\usepackage{tabularx}
\usepackage{mathabx}
\usepackage[all]{xy}

\newcommand{\dlpullback}[1][dl]{\save*!/#1-4ex/#1:(-1,1)@^{|-}\restore}

\newcommand{\drpullback}[1][dr]{\save*!/#1-4ex/#1:(-1,1)@^{|-}\restore}
\newdir{ >}{{}*!/-7pt/@{>}}

\newtheoremstyle{mystyle}
  {}
  {}
  {\itshape}
  {\parindent}
  {\bfseries}
  {.}
  {.5em}
  {}
\theoremstyle{mystyle}

\newtheorem{prop}[subsection]{Proposition}
\newtheorem{lemma}[subsection]{Lemma}
\newtheorem{cor}[subsection]{Corollary}

\theoremstyle{remark}

\newtheorem{taller}[subsection]{$\!\!$}

\newenvironment{blanko}[1]%
{\begin{taller}{\normalfont\bfseries  #1}\normalfont}%
{\end{taller}}

\newcommand{\ourX}{D}

\newcommand{\simplexcategory}{\triangle\!}

\newcommand{\kat}[1]{\text{\textbf{\textsl{#1}}}}
\newcommand{\Set}{\kat{Set}}
\newcommand{\Cat}{\kat{Cat}}
\newcommand{\Grpd}{\kat{Grpd}}
\newcommand{\grpd}{\kat{grpd}}
\newcommand{\relfin}{\mathrm{rel.fin.}}

\newcommand{\Id}{\operatorname{Id}}
\newcommand{\Aut}{\operatorname{Aut}}
\newcommand{\Map}{\operatorname{Map}}

\newcommand{\upperstar}{^{\raisebox{-0.25ex}[0ex][0ex]{\(\ast\)}}}
\newcommand{\lowerstar}{_{\raisebox{-0.33ex}[-0.5ex][0ex]{\(\ast\)}}}
\newcommand{\lowershriek}{_!}

\newcommand{\tensor}{\otimes}
\newcommand{\op}{^{\text{{\rm{op}}}}}

\newcommand{\Q}{\mathbb{Q}}
\newcommand{\N}{\mathbb{N}}
\newcommand{\B}{\mathbb{B}}
\newcommand{\C}{\mathbb{C}}
\newcommand{\T}{\mathbb{T}}
\newcommand{\G}{\mathbb{G}}

\providecommand{\norm}[1]{\left| {#1}\right|}

\newcommand{\RRR}{\mathsf{R}}
\newcommand{\SSS}{\mathsf{S}}

\newcommand{\name}[1]{\ulcorner #1 \urcorner}

\newcommand{\LIN}{\kat{LIN}}

\begin{document}

\title
[Fa\`a di Bruno for operads]
{Fa\`a di Bruno for operads and internal algebras
}

\author{Joachim Kock}
\address{Departament de matem\`atiques, Universitat Aut\`onoma de Barcelona}
\email{kock@mat.uab.cat}
\author{Mark Weber}
\address{Institute of Mathematics of the Czech Academy of Sciences}
\email{mark.weber.math@gmail.com}

\subjclass[2010]{16T10, 05A19, 18D50, 18C15, 18B40, 32A05, 57T30, 18G30}

\begin{abstract}
  For any coloured operad $\RRR$, we prove a Fa\`a di Bruno formula for the
  `connected Green function' in the incidence bialgebra of $\RRR$.
  This generalises on one hand the classical Fa\`a di Bruno formula (dual to
  composition of power series), corresponding to the case where $\RRR$ is the
  terminal reduced operad, and on the other hand the Fa\`a di Bruno formula for
  $P$-trees of G\'alvez--Kock--Tonks ($P$ a finitary polynomial endofunctor),
  which corresponds to the case where $\RRR$ is the free operad on $P$.  Following
  G\'alvez--Kock--Tonks, we work at the objective level of groupoid slices, hence
  all proofs are `bijective': the formula is established as the homotopy
  cardinality of an explicit equivalence of groupoids.  
  In fact we establish the formula more generally in a relative situation, for
  algebras for one polynomial monad internal to another.  This covers in 
  particular
  nonsymmetric operads (for which the terminal reduced case yields the
  noncommutative Fa\`a di Bruno formula of Brouder--Frabetti--Krattenthaler).
\end{abstract}

\maketitle
\tableofcontents
\setcounter{section}{-1}

\section{Background}

The Fa\`a di Bruno formula computes the coefficients of the composite of
two formal power series without constant terms.  For
$$
f(z)= \sum_{n=1}^\infty f_n \frac{z^n}{n!} \qquad \text{ and } \qquad 
g(z)= \sum_{n=1}^\infty g_n \frac{z^n}{n!}
$$
the composite series
$$
 (g\circ f)(z) =: \sum_{n=1}^\infty h_n \frac{z^n}{n!}
\qquad \text{ has } \qquad
h_n = \sum_{k=1}^n B_{n,k}(f_1,f_2,\ldots) \cdot g_k
$$
where $B_{n,k}(x_1,x_2,x_3,\ldots)$ are the partial Bell polynomials,
whose coefficient of a monomial $x_1^{\lambda_1} x_2^{\lambda_2} \cdots 
x_n^{\lambda_n}$ counts the number of partitions of an $n$-element
set into $k$ blocks: $\lambda_1$ blocks of size $1$, $\lambda_2$ blocks of size $2$, etc.
($\sum_i \lambda_i = k$ and $\sum_i i \lambda_i = n$).

%
%
%
%
%
%
We shall not actually need this classical formulation, for which we refer to
the survey of Figueroa and Gracia-Bond\'ia~\cite{Figueroa-GraciaBondia:0408145}.
An excellent historical account of the formula 
is due to Johnson~\cite{Johnson:MR1903577}.

Our starting point is instead an elegant coalgebraic rendition of the formula,
which, as far as we know, was first noticed by Brouder, Frabetti and
Krattenthaler~\cite{Brouder-Frabetti-Krattenthaler:0406117}, inspired by
constructions in perturbative quantum field theory.  Let $A_k$ denote the linear
functional on the vector space of power series that takes $f$ to $f^{(k)}(0)$,
i.e.~returns the coefficient of $z^k/k!$.  The {\em Fa\`a di Bruno bialgebra} is the
polynomial ring
$\C[A_1,A_2,A_3,\ldots]$ with comultiplication dual to the
monoid structure of composition of power series. Doubilet~\cite{Doubilet:1974} proved
that the Fa\`a di Bruno bialgebra is the reduced incidence bialgebra of the 
lattice of set partitions, and Joyal~\cite{Joyal:1981}
observed that it can also be obtained directly
(without a reduction step) from the category of surjections.

Brouder, Frabetti and
Krattenthaler introduce the infinite series
$$
A = \sum_{k=1}^\infty A_k / k!  \in \C[[A_1,A_2,A_3,\ldots]],
$$
check that the comultiplication extends to the power-series ring, and show that 
the Fa\`a di Bruno
formula can be formulated succinctly as
\begin{equation}\label{eq:FdB}
  \Delta(A) = \sum_{k=1}^\infty A^k \tensor A_k/k!
\end{equation}
(The exponent $k$ is 
$k$th power in the ring.)
The individual coefficients can be extracted from this formula.

This form of the Fa\`a di Bruno formula is of importance in quantum field
theory, where the role of the series $A$ is played by the {\em connected Green
function} \cite{Bergbauer-Kreimer:0506190}, defined in the Connes--Kreimer Hopf
algebra of Feynman graphs \cite{Connes-Kreimer:9912092} as the sum of all 
connected graphs
divided by their symmetry factors.  Van Suijlekom~\cite{vanSuijlekom:0807} established a
(multivariate) Fa\`a di Bruno formula for the connected Green function, thereby
vindicating the relevance of the Hopf algebra of graphs also in non-perturbative
QFT.
The coalgebraic Fa\`a di Bruno formula \eqref{eq:FdB}
has also been exploited in the so-called
exponential renormalisation~\cite{EbrahimiFard-Patras:1003.1679}.  Inspired by
van Suijlekom's result, G\'alvez, Kock and
Tonks~\cite{GalvezCarrillo-Kock-Tonks:1207.6404} proved a general Fa\`a di Bruno
formula in bialgebras of $P$-trees ($P$ a finitary polynomial endofunctor),
introducing categorical and homotopical methods which we further exploit in the
present contribution to prove a Fa\`a di Bruno formula in a much more general
setting.

Beyond calculus, combinatorics, and classical applications to probability 
theory (for the latter, see for example \cite{Kendall:vol1}), 
Fa\`a di Bruno-type formulae pop
up in various contexts in the mathematical 
sciences.  Most prominently perhaps in algebraic topology, in connection with
complex cobordism~\cite{
Morava:Oaxtepec} and vertex operator algebras~\cite{Robinson:0903.3391},
but also in areas such as
control theory 
\cite{DuffautEspinosa-EbrahimiFard-Gray:1406.5396, Gray-DuffautEspinosa:MR2849486},
population genetics~\cite{Hoppe2008543},
and differential linear logic~\cite{Cockett-Seely:FdB}, to mention a few that 
have come to our attention.
We do not know whether our contribution can be of any relevance in these 
contexts.

\normalsize

\section{Outline of results and proof ingredients}

\begin{blanko}{Heuristic outline.}
  Given an operad $\RRR$ (satisfying finiteness conditions, cf.~\ref{fin-cond} 
  below), one can form its incidence bialgebra:  as an 
  algebra it is the polynomial ring in
  the set of iso-classes of operations of $\RRR$.  The monomials are interpreted as
  formal disjoint unions of operations.  The comultiplication is the crucial
  structure: an operation is comultiplied by summing over all ways the operation
  can arise by operad substitution from a collection of operations fed into a
  single operation:
  $$
  \Delta(r) = \sum_{r=b \circ (a_1,\ldots,a_n)}  a_1\cdots a_n \tensor b  .
  $$
  (The sum is over iso-classes of factorisations---to be technically correct it
  should rather be a homotopy sum, as will be detailed.)
  The comultiplication is extended multiplicatively to the
  whole polynomial ring, which therefore becomes a bialgebra, the {\em incidence
  bialgebra} of the operad~\cite{GKT:ex}.
  
  Inside the completion of this bialgebra (the power series ring), we now define
  the {\em connected Green function} (by analogy with quantum 
  field theory) to be the series consisting of all the operations
  themselves (but not their disjoint unions---this is the meaning of the word 
  `connected'), divided by their symmetry factors.
  This series $G$ can be written as an infinite sum
  $$
  G = \sum_n g_n
  $$
  where $g_n$ consists of all the operations of arity $n$ ($n$ a sequence of 
  colours), divided by their 
  symmetry factors.  The
  comultiplication is shown to extend to the power series ring, and we can
  now state our general Fa\`a di Bruno formula:
  
  \bigskip
  
  \noindent {\bf Main Theorem.} (Cf.~\ref{thm:main}.)
  $$
  \Delta(G) = \sum_k G^k \tensor g_k .
  $$ 
  
  \bigskip
  
  The two extreme examples of this construction are the following.
  When $\RRR$ is the terminal reduced operad $\mathsf{Comm}_+$ (i.e.~no nullary 
  operations, and a single $n$-ary operation for each $n>0$),
  then the formula is the classical \eqref{eq:FdB} (cf.~Example~\ref{ex:clas}),
  with $g_k$ denoting the unique
  operation in arity $k$ (divided by $k!$), corresponding to $A_k/k!$.
  When $\RRR$ is the free operad on a finitary polynomial 
  endofunctor $P$, then the operations are the $P$-trees, and $g_n$ is given by 
  iso-classes of $P$-trees with $n$ leaves.  The resulting Fa\`a di Bruno 
  formula in this case is that of 
  G\'alvez--Kock--Tonks~\cite{GalvezCarrillo-Kock-Tonks:1207.6404} 
  (cf.~Example~\ref{ex:free}).
  More examples are given in Section~\ref{sec:ex}.
\end{blanko}

\begin{blanko}{Formalisation.}
  It is likely that the formula could be proved by hand, just by expansion of
  series and brute computation.  The difficulty in that approach is to handle
  correctly the symmetry factors that appear, as well illustrated by the
  computations of van Suijlekom \cite{vanSuijlekom:0610137, vanSuijlekom:0807}.
  The insight of
  \cite{GalvezCarrillo-Kock-Tonks:1207.6404} in the case of trees was that the
  formula can be realised as the homotopy cardinality of an equivalence of
  groupoids (hence constituting a bijective proof), and that at the groupoid level all the symmetry factors take care
  of themselves automatically. 
  The actual equivalence established in
  \cite{GalvezCarrillo-Kock-Tonks:1207.6404} involves groupoids of trees and
  trees with a cut, relying on
  specifics of the combinatorics of the
  Butcher--Connes--Kreimer bialgebra of trees.
  
  The present contribution exploits the objective method initiated in
  \cite{GalvezCarrillo-Kock-Tonks:1207.6404}, establishing the Fa\`a di Bruno
  formula as the homotopy cardinality of an equivalence of groupoids, but takes
  a further abstraction step, which leads to a more general formula and a much
  simpler proof.  We achieve this by leveraging some recent advances in category
  theory: on one hand, $2$-categorical perspectives on operads and related structures
  discussed in \cite{Weber-OpPoly2Mnd, Weber-CodescCrIntCat, Weber:Tbilisi},
  and on the other hand, the theory of decomposition
  spaces~\cite{Galvez-Kock-Tonks:1512.07573,Galvez-Kock-Tonks:1512.07577}.  With
  these tools, the proof of the equivalence of groupoids ends up being
  rather neat, emerging naturally from general principles.
\end{blanko}

\begin{blanko}{The objective method in a nutshell: groupoid slices instead of vector spaces.}
  It is well appreciated in combinatorics that bijective proofs represent
  deeper insight than formal algebraic ones.  From the viewpoint of category
  theory, this deeper insight typically involves universal properties.  Algebraic
  objects associated to combinatorial structures often have underlying vector
  spaces, generated freely by isomorphism classes of the combinatorial objects.
  It is a general hypothesis that whenever linear combinations involve
  symmetries---as exemplified in quantum field theory where formal sums 
  of Feynman graphs are
  weighted by inverses of symmetry factors---they arise as homotopy
  cardinalities of groupoids rather than just cardinalities of sets; more
  precisely, as homotopy cardinalities of homotopy sums, as we recall below.
  Accordingly, algebraic identities in such situations should reflect
  equivalences of groupoids rather than bijections of sets.
  
  A systematic way of expressing algebraic identities objectively is
  to replace vector spaces by slice categories.  If $B$ is a groupoid of 
  combinatorial objects, then the basic vector space of interest is the free vector
  space on $\pi_0 B$, the set of iso-classes of objects in $B$.
  Just as a vector is a formal linear combination of elements 
  in $\pi_0 B$ (i.e.~a family of scalars indexed by $\pi_0 B$), an element $p:X\to B$ in the
  slice category $\Grpd_{/B}$ is a collection of groupoids indexed by $B$,
  the members of the family being the (homotopy) fibres $X_b$.  Similarly,
  linear maps are represented by `linear functors', in turn given by spans of 
  groupoids.  The ordinary linear algebra is obtained by taking the homotopy cardinality of
  the groupoid-level `linear algebra' \cite{Galvez-Kock-Tonks:1602.05082}.
  The homotopy cardinality of a family $p:X\to B$ in $\Grpd_{/B}$ is defined as
  $$
  \norm{p} = \sum_{b\in \pi_0 B} \frac{\norm{X_b}}{\norm{\Aut{b}}} \; \delta_b
  $$
  which is an element in the vector space $\Q_{\pi_0 B}$ spanned by the symbols
  $\delta_b$, one for each iso-class of objects in $B$.
%
\end{blanko}

\begin{blanko}{Brief explanation of the groupoid equivalence.}
  To explain the equivalence briefly, for simplicity we take the case where $\RRR$ is 
  single-coloured, and gloss over a few technical details.
  The basic vector space is that spanned by the monomials
  in the (iso-classes of) operations of the operad $\RRR$.  Accordingly, we shall work 
  with the groupoid $\ourX_1 = \SSS C_1$, the free symmetric monoidal category on the 
  groupoid $C_1$ of operations of $\RRR$; more precisely, $C_1$ is the action 
  groupoid (i.e.~homotopy quotient) for the action of the symmetric groups on the sets of operations.
  So the basic slice is $\Grpd_{/\ourX_1}$ whose cardinality is the 
  vector space
  $\Q_{\pi_0 \ourX_1}$ spanned by $\pi_0 \ourX_1$.
  The {\em connected Green function} $G$ is the formal series defined as the homotopy 
  cardinality of the object $C_1 \to \ourX_1$ in
  $\Grpd_{/\ourX_1}$.  The comultiplication map $\Delta$ is the cardinality of a
  certain span~\cite{Galvez-Kock-Tonks:1512.07573},
  $$
  \ourX_1 \stackrel{d_1}\longleftarrow \ourX_2 \stackrel{(d_2,d_0)}\longrightarrow \ourX_1 \times \ourX_1
  $$
  whose maps refer to a simplicial groupoid $\ourX_\bullet:\simplexcategory\op\to \Grpd$ 
  canonically constructed from the operad as a certain relative (two-sided) bar
  construction (cf.~Section~\ref{sec:bar}).  The Fa\`a di Bruno
  formula, which at the algebraic level is an equation in the vector space
  $\Q_{\pi_0 \ourX_1} \tensor \Q_{\pi_0 \ourX_1}$, is therefore supposed to be the
  cardinality of an equivalence of groupoids over $\ourX_1 \times \ourX_1$.  Here is the
  equivalence (established in Proposition~\ref{prop:equiv}):
  $$\xymatrix @!C=24pt {
     C_1 \times_{\ourX_1} \ourX_2  \ar[dr] &\simeq & \hspace{12pt}\int^k (\ourX_1)_k \times {}_k(C_1) \ar[ld] \\
     & \ourX_1 \times \ourX_1   . &
  }$$
  The left-hand side is precisely the definition of the comultiplication of the
  connected Green function $C_1 \to \ourX_1$, so its homotopy cardinality is
  $\Delta(G)$.  The right-hand side is a groupoid $\ourX_1 \times_{\ourX_0}
  C_1$, written as the homotopy sum of its homotopy fibres over $\ourX_0$, the
  groupoid of finite sets and bijections, and it has to be unravelled a bit:
  $(\ourX_1)_k$ denotes the $k$-component of the groupoid of families of
  operations, meaning those families that have $k$ members.  This is the same
  thing as $k$-tuples of operations, $(\ourX_1)_k \simeq (C_1)^k$, whose
  homotopy cardinality is precisely $G^k$.  The symbol ${}_k (C_1)$ denotes the
  homotopy fibre of $C_1$ over $k$ under the arity map, so it is the groupoid of
  $k$-ary operations, and maps fixing the input slots.  The integral sign
  designates a homotopy sum \cite{Carlier-Kock, Galvez-Kock-Tonks:1602.05082},
  $$
  \int^k ( \ ) = \sum_k \frac{( \ )}{\Aut{k}}
  $$
  where the division bar denotes homotopy quotient under the action of $\Aut(k)$,
  acting diagonally on $(\ourX_1)_k \times {}_k(C_1)$.  Since the action on the
  first factor is free, the action can be passed to the second factor, and the
  effect of it is to add morphisms to the groupoid ${}_k(C_1)$ to allow  
  permutation of the inputs.  Altogether, $\frac{{}_k(C_1)}{\Aut(k)}$ is the full
  groupoid of $k$-ary operations, and its homotopy cardinality is precisely $g_k$,
  showing that altogether the right-hand side of the equivalence has homotopy 
  cardinality
  $$
  \sum_k G^k \tensor g_k
  $$
  as claimed.
  
  The task is now to define the involved groupoids $\ourX_i$ and $C_i$ 
  correctly.
  It is a pleasing aspect of our approach that these groupoids come about by 
  the standard general construction in algebraic topology
  known as the relative (two-sided) bar construction.
\end{blanko}

\begin{blanko}{Finiteness conditions.}\label{fin-cond}
  At the groupoid level, in the `objective' bialgebra $\Grpd_{/\ourX_1}$, the Fa\`a
  di Bruno formula holds for {\em any} operad.  However, in order to be able to
  take cardinality, a certain finiteness condition must be imposed
  \cite{Galvez-Kock-Tonks:1512.07577} (which is automatic in the two previously
  known cases, when either $\RRR$ is $\mathsf{Comm}_+$ or free on a polynomial
  endofunctor).  Namely, an
  operad $\RRR$ is called {\em locally finite} when for each operation $r$, the
  groupoid of possible decompositions $r = b\circ (a_1,\ldots,a_n)$ is homotopy
  finite.  Equivalently, the map $d_1: \ourX_2 \to \ourX_1$ is homotopy finite.
\end{blanko}

\begin{blanko}{Remark on grading and Hopf versus bialgebras.}\label{rmk:grading}
  The classical Fa\`a di Bruno bialgebra is naturally graded, with $\deg A_k = k-1$. 
  Because of this minus-one, a shifted indexing convention is often used in the
  literature \cite{Brouder-Frabetti-Krattenthaler:0406117, EbraihimiFard-Lundervold-Manchon:1402.4761, vanSuijlekom:0610137, vanSuijlekom:0807}, writing $A_{k-1}$
  instead of $A_k$.  The present convention (also that of
  \cite{GalvezCarrillo-Kock-Tonks:1207.6404} and \cite{Kock:1512.03027}) is
  dictated by the operadic approach, where it is essential that the superscript
  on $G^k$ (counting $k$ outputs) matches the subscript on $g_k$ (counting $k$
  inputs).
  
  It is also worth noting that it is essential to work with bialgebras rather
  than Hopf algebras.  Our bialgebras are not connected, as in degree zero they
  are certain free symmetric monoidal categories, such as in the classical case
  the groupoid of finite sets and bijections (to which $k$ belongs).  Hopf
  algebras can be obtained by collapsing degree zero,
  but this amounts to throwing away the data controlling
  the match, as just described.
\end{blanko}

\begin{blanko}{Generalisations.}
  The arguments actually work the same for any situation $\RRR \Rightarrow
  \SSS$, of two polynomial monads, one cartesian over the other, but possibly
  over different slices.  The construction also works for $\RRR$-algebras
  rather than for $\RRR$ itself (which is actually the case of the terminal
  $\RRR$-algebra).
  We describe these generalisations in
  Section~\ref{sec:generalised}. 
  
  In the operad case, $\SSS$ is the
  symmetric monoidal category monad.  The comonoid structure comes from the
  combinatorics of $\RRR$, and it will (almost) always be noncocommutative, as a
  consequence of the fact that operads have many inputs but only one output.
  The algebra structure, which is always free, is of a different nature,
  deriving from the fact that operads are considered internal to symmetric
  monoidal categories, by the choice of $\SSS$.  For general $\SSS$, the outcome
  will not exactly be a bialgebra, but rather a free $\SSS$-algebra in the
  category of coalgebras.  Some care is needed to interpret the Fa\`a di Bruno
  formula correctly in this more general setting, described in
  Section~\ref{sec:generalised}.
\end{blanko}

%
%
%

\section{Operads and polynomial monads}

\label{sec:operads}

Our main interest is in operads, but it is
technically convenient to deal with them in the setting of polynomial monads,
which also leads to a natural generalisation.
The natural level of generality for our Fa\`a di Bruno formula
is that of one polynomial monad cartesian over another
$$
\RRR\Rightarrow \SSS 
$$
in a double-category sense~\cite{Fiore-Gambino-Kock:1006.0797}.
However, for expository reasons, and since it is the main case, we concentrate
on the case of operads, namely when $\SSS$ is the symmetric monoidal category
monad (\ref{S}), which we assume in Sections~\ref{sec:operads}--\ref{sec:card}.
Then in Section~\ref{sec:generalised} we explain how
everything carries over readily to the general case.

\begin{blanko}{Groupoids and homotopy sums.}\label{hosum}
  We freely use basic homotopy theory of groupoids, such as homotopy pullbacks 
  and homotopy fibres, referring to \cite{Carlier-Kock} for all details.
  Here we content ourselves to briefly review the notion of homotopy sum, since
  it is a key point, accounting for the origin of the symmetry factors,
  as first exploited in \cite{GalvezCarrillo-Kock-Tonks:1207.6404}.
  
  It is plain that for a map of sets $E \to B$, the set $E$ can be regarded as the sum of its
  fibres, $E \simeq \sum_{b\in B} E_b$.  The same is true for groupoids, 
  provided we use homotopy fibres and homotopy sums, as we now recall.
  Each homotopy fibre $E_b$ comes with a canonical map to $E$, but it is not
  fully faithful.  But the automorphism group $\Aut(b)$ acts on $E_b$ 
  canonically, and the action groupoid (also called homotopy quotient)
  $E_b/\Aut(b)$ does map to $E$ fully 
  faithfully; summing over one element $b$ for each connected component in $B$
  then yields an equivalence of groupoids
  $$
  \sum_{b \in \pi_0 B} \frac{E_b}{\Aut(b)} \simeq E .
  $$
  The left-hand side is an example of a homotopy sum, and is denoted 
  $\int^{b\in B} E_b$.  It is an instance of a homotopy colimit, indexed
  by the groupoid $B$, just as an ordinary sum is a colimit indexed by a set.
  (The integral notation is standard, and is also compatible with
  general usage in category theory, since it is also an instance of a coend.)
  The great benefit of working with homotopy sums, is that they interact with 
  homotopy pullbacks in the nicest way, precisely as ordinary sums interact
  with pullbacks in the category of sets. 
  
  The following lemma, which is a 
  straightforward
  variation of this splitting into fibres, will be crucial for the Fa\`a di 
  Bruno formula.
\end{blanko}

 \begin{lemma}\label{lem:XSY}
  Given a (homotopy) pullback square
  $$\xymatrix{
     P \drpullback \ar[r]\ar[d] & Y \ar[d] \\
     X \ar[r] & S
  }$$
  there is a natural equivalence of groupoids
  $$
  P \simeq \int^{s\in S} X_s \times Y_s .
  $$
  It is also an equivalence over $X \times Y$, meaning more precisely an
  equivalence in the weak slice $\Grpd_{/X\times Y}$.
\end{lemma}

\begin{blanko}{Polynomial monads, classically.}
  The theory of polynomial functors has
  roots in topology, representation theory, combinatorics, logic, and computer
  science.  A standard reference is \cite{GambinoKock-PolynomialFunctors}, which
  also contains pointers to those original developments.

  A {\em polynomial} is a diagram of sets
  $$
  I \stackrel s\leftarrow E \stackrel p\to B \stackrel t\to I'  .
  $$
  It defines a polynomial functor $\Set{/I} \to \Set{/I'}$ by the formula
  $$
  t\lowershriek \circ p\lowerstar  \circ s\upperstar  .
  $$
  Polynomial functors form a double category in which the
  $2$-cells are diagrams of the form
  $$\xymatrix{
    \cdot \ar[d] & \ar[l] \cdot \drpullback \ar[r]\ar[d] & \cdot \ar[d] \ar[r] 
    & \cdot \ar[d] \\
    \cdot & \ar[l] \cdot \ar[r] & \cdot \ar[r] & \cdot
  }$$
  Polynomial monads are horizontal monads in this double category,
  and the relevant monad maps are the vertical monad 
  maps~\cite{Fiore-Gambino-Kock:1006.0797}.
  At the level of functors, this situation amounts to having one monad $\RRR$ on the slice
  category $\Set{/I}$, another monad $\SSS$ on the slice category $\Set{/J}$,
  and a map $F: I \to J$ for which $F\lowershriek: \Set{/I}\to \Set{/J}$ and 
  its right adjoint $F\upperstar$ form a
  monad adjunction in the sense of \cite{Weber-CodescCrIntCat}.  This is turn amounts to
  having a natural transformation $\phi : F\lowershriek \RRR \to \SSS F\lowershriek$
  making $(F\lowershriek, \phi)$ into a monad opfunctor in the sense of 
  Street~\cite{Street:formal-monads}.
  
  Polynomial monads over
  $\Set$ can account for nonsymmetric
  operads~\cite{GambinoKock-PolynomialFunctors} and more generally sigma-free
  operads~\cite{Kock:0807}, but to account for general (symmetric) operads, at
  least groupoids are needed instead of sets.
\end{blanko}

\begin{blanko}{Polynomial functors over groupoids.}
  Polynomial functors over groupoids can be dealt with either in a homotopical
  setting as outlined in \cite{Kock:1210.0828} or in a $2$-categorical
  setting~\cite{Weber-PolynomialFunctors}.  
  
  In the homotopical setting, the
  involved groupoids are only ever defined up to homotopy equivalence; one works
  with weak slices, and all notions are homotopy, e.g.~homotopy pullbacks,
  homotopy fibres, etc., exploiting that in the homotopy sense groupoids form a
  locally cartesian closed category.  Ultimately, the natural setting for this
  approach is that of $\infty$-groupoids, as adopted in 
  \cite{Galvez-Kock-Tonks:1512.07573,Galvez-Kock-Tonks:1512.07577}.  We follow
  `tradition' in homotopy theory of denoting the weak slices $\Grpd_{/I}$,
  with a subscripted slash.
  
  On the other hand, in the $2$-categorical approach one works mostly with
  strict slices $\Grpd/I$ (for which we use non-subscripted slashes), strict
  pullbacks, and so on.  (Actually the $2$-categorical approach deals naturally
  with categories instead of groupoids, as exploited to good effect in
  \cite{Weber-PolynomialFunctors,Weber-OpPoly2Mnd,Weber-CodescCrIntCat}, but for
  the present purposes we stick with groupoids.)  The $2$-category $\Grpd$ is
  not locally cartesian closed, so one has to assume the middle maps in the
  polynomial diagrams to be fibrations, in order for pullbacks to have right
  adjoints (`lowerstars').  Fibrations also play an important role to ensure
  that various `strict' constructions involving pullbacks are homotopically
  meaningful.  Such issues are handled efficiently in terms of the {\em
  fibration monads} \cite{Street-FibrationIn2cats}, which turns general
  functors into fibrations.  The weak slice $\Grpd_{/I}$
  appears as the Kleisli category for
  the fibrations monad on the strict slice $\Grpd/I$, and the homotopy content is
  essentially controlled in terms of compatibilty with the fibration monads
  \cite{Weber:TAC18,Weber-OpPoly2Mnd,Weber-CodescCrIntCat}.
  
  Both approaches will be exploited here: for the purpose of setting up the
  simplicial groupoid $\ourX_\bullet$, we shall work $2$-categorically, since
  it gives more precise results, and since a well-developed theory is already 
  in place.  But once that simplicial groupoid is in place,
  we shall pass to the homotopical setting, since the main object of interest
  is the weak slice $\Grpd_{/\ourX_1}$.
\end{blanko}

\begin{blanko}{Key example: the symmetric monoidal category monad.}\label{S}
  From now on, and until Section~\ref{sec:generalised}, $\SSS:\Grpd \to \Grpd$
  denotes the symmetric monoidal category monad.  It is polynomial,
  represented by the polynomial
  $$
  1\leftarrow \B' \to \B \to 1
  $$
  where $\B$ is the groupoid of finite sets and bijections, and $\B'$ is the
  groupoid of finite pointed sets and basepoint-preserving bijections. 
  The formula for evaluation is
  $$
  X \mapsto \int^{n\in \B} \Map(\B'_n,X) \simeq \int^{n\in \B} X^{\underline n},
  $$
  where $\underline n$ denotes the fibre over $n$.
  Note
  that $\B'\to\B$ is a discrete fibration (in fact the classifier for finite discrete
  maps of groupoids).  Hence any monad cartesian over $\SSS$ will
  automatically be finitary again, meaning that all operations have
  finite discrete arity.
\end{blanko}

\begin{blanko}{Operads as polynomial monads.}\label{op=mnd}
  By operad we mean coloured symmetric operad in the category of sets.  It was
  shown in \cite{Weber-OpPoly2Mnd} (Theorem~3.3)
  that operads are the same thing as polynomial
  monads cartesian over the symmetric monoidal category monad
  $$\xymatrix{
     I \ar[d] & \ar[l] E\drpullback \ar[r]\ar[d] & B\ar[d]\ar[r] & I \ar[d] \\
     1 & \ar[l] \B'\ar[r] & \B\ar[r] & 1}$$
  for which $I$ is a set, and $B \to \B$ is a discrete fibration.
  The polynomial viewpoint on operads has proven very useful in homotopy 
  theory~\cite{Batanin-Berger:1305.0086}.
  
  The equivalence goes as follows.  Given an operad $\RRR$,
  let $I$ be its set of colours.  
  Let $B$ be the action groupoid (i.e.~homotopy quotient) of
  the action of the symmetric groups on the
  sets of operations of each arity.  More precisely, for the symmetric-group
  action $\RRR_n \times \mathfrak S_n \to \RRR_n$, the action groupoid has as
  objects the elements in $\RRR_n$, and an arrow from $r$ to $r'$ for each $g\in
  \mathfrak S_n$ such that $r.g = r'$.  Finally, $B$ is the disjoint union of
  these action groupoids, with its canonical map to $\B$, itself the disjoint
  union of the classifying spaces of the $\mathfrak S_n$.  This is a discrete
  fibration, whose fibre over $n$ is the set $\RRR_n$, by the standard fibre
  sequence for action groupoids~\cite{Carlier-Kock}.
  Note that the action respects colours; for
  this reason it is meaningful to define the map $B \to I$ by assigning to an
  operation its output colour.  The groupoid $E$ consists of operations with a
  marked input slot.  The map $E \to I$ assigns to a marked operation the colour
  of its marked input.  The fibre of $E \to B$ over an operation $r$ is the set
  of its input slots. 
  The monad structure
  on the polynomial endofunctor comes precisely from the substitution operation
  of the operad $\RRR$.  The fact that this monad on $\Grpd/I$ is itself polynomial,
  implies that it is cartesian.
  
  Conversely, given a polynomial monad cartesian over $\SSS$ as above, the
  discrete fibration $B \to \B$ induces a $\mathfrak S$-set which is the set of
  operations.  The set of $n$-ary operations is the (homotopy) fibre of $B \to
  \B$ over $n$. 
  The fact that there is a pullback square gives the set of
  input slots
  a linear order (interpreting $\B$ as the groupoid of linear orders $\underline
  n = \{1,2,\ldots,n\}$ and not-necessarily-monotone maps), as befits an operad 
  in the classical sense. 
  The operad substitution comes from the monad multiplication.

  Although not central to this article, the monads $\RRR$, $\SSS$ and
  the fibrations monads described above, are all \emph{$2$-monads}. In particular
  since $\RRR$ is a $2$-monad, one may consider its strict algebras, and from
  \cite{Weber-OpPoly2Mnd} these were understood to be `weakly-equivariant'
  $\Cat$-valued algebras of the operad $\RRR$. The $2$-dimensionality of $\RRR$ and $\SSS$
  enables the definition of an `algebra of $\RRR$ internal to an algebra of $\SSS$'
  to be made, as in
  \cite{Batanin:0207281, Batanin-Berger:1305.0086, Weber-CodescCrIntCat},
  and it is these that correspond to
  the algebras of the operad $\RRR$ in the usual sense.  
  From now on, the symbol $\RRR$ refers rather to the polynomial monad, which is
  what we actually work with.
\end{blanko}

\section{Monad adjunctions and the relative ($2$-sided) bar construction}

\label{sec:bar}

\begin{blanko}{Set-up.}
  We consider the situation as in \ref{op=mnd}, where we have a polynomial
  monad $\RRR$ on $\Grpd{/I}$, cartesian over the
  symmetric monoidal category monad $\SSS$ on $\Grpd$,
  represented altogether by the polynomial diagram
    $$\xymatrix{
     I \ar[d]_F & \ar[l] E\drpullback \ar[r]\ar[d] & B\ar[d]\ar[r] & I \ar[d]^F \\
     1 & \ar[l] \B'\ar[r] & \B\ar[r] & 1 
  }$$
  The two monads $\RRR$ and $\SSS$ are intertwined by means of the functor
  $F\lowershriek$ and the natural transformation
  $$
  \phi : F\lowershriek \RRR \Rightarrow \SSS F\lowershriek 
  $$
  forming together a monad opfunctor in the sense of 
  Street~\cite{Street:formal-monads}, a monad adjunction in the sense of 
  \cite{Weber-CodescCrIntCat},
  or a vertical morphism of horizontal monads in the double-category 
  setting~\cite{Fiore-Gambino-Kock:1006.0797}.
  
  It is a general fact that the
  natural transformation $\phi$ is cartesian.  This follows
  because its ingredients are the unit and counit of lowershriek-upperstar 
  adjunctions and an instance of the Beck--Chevalley isomorphism.  See 
  \cite[\S~3.3]{Weber-CodescCrIntCat} for details.
\end{blanko}
  
\begin{blanko}{A simplicial-object-with-missing-top-face-maps.}
  Let $A$ denote the terminal object in $\Grpd/I$, and let $\alpha: \RRR A \to 
  A$ denote the unique map from $\RRR A$ in $\Grpd/I$. 
  (Note that $\alpha$ has underlying map of groupoids $B \to I$.)
  The pair $(A,\alpha)$ is
  the terminal $\RRR$-algebra.

  In $\Grpd/I$ there is induced a natural 
  simplicial-object-with-missing-top-face-maps
  \begin{equation}\label{eq:RA}
  \xymatrix@C+1.8em{
  A 
  \ar[r]|(0.6){s_0} 
  &
  \ar[l]<+2mm>|(0.6){d_0}\ar@{}[l]<-2.3mm>|{\cdot \  \cdot \ \cdot \ \cdot \ \cdot\ \cdot\ \cdot } 
  \RRR A 
  \ar[r]<-2mm>|(0.6){s_0}\ar[r]<+2mm>|(0.6){s_1}  
  &
  \ar[l]<+4mm>|(0.6){d_0}\ar[l]|(0.6){d_1}\ar@{}[l]<-4.3mm>|{\cdot \ \cdot \ \cdot \ \cdot \ \cdot\ \cdot\ \cdot }
  \RRR\RRR A
  \ar[r]<-4mm>|(0.6){s_0}\ar[r]|(0.6){s_1}\ar[r]<+4mm>|(0.6){s_2}  
  &
  \ar[l]<+6mm>|(0.6){d_0}\ar[l]<+2mm>|(0.6){d_1}\ar[l]<-2mm>|(0.6){d_2}\ar@{}[l]<-6.3mm>|{\cdot \  \cdot \ \cdot \ \cdot \ \cdot\ \cdot\ \cdot }
  \RRR\RRR\RRR A &
  \ar@{}|\cdots[r]
  &
  }  
  \end{equation}
  The bottom face maps come from the action $\alpha: \RRR A \to A$,
  and the remaining face and degeneracy maps come from the monad structure.
  Hence the basic maps are 
  $$
  \xymatrix@C+3em{
  A 
  \ar[r]|(0.6){\eta_A} 
  &
  \ar[l]<+2mm>|(0.6){\alpha}\ar@{}[l]<-2.4mm>|{\cdot \  \cdot \ \cdot \ \cdot \ \cdot\ \cdot\ \cdot } 
  \RRR A 
  \ar[r]<-2mm>|(0.6){\RRR\eta_A}\ar[r]<+2mm>|(0.6){\eta_{\RRR A}}  
  &
  \ar[l]<+4mm>|(0.6){\RRR\alpha}\ar[l]|(0.6){\mu_A}\ar@{}[l]<-4.4mm>|{\cdot \ \cdot \ \cdot \ \cdot \ \cdot\ \cdot\ \cdot }
  \RRR\RRR A
  }$$
  and in general, $s_k  : \RRR^n A \to  \RRR^{n+1}A$  is given by $\RRR^{n-k} 
  \eta_{\RRR^k 
  A}$ and
  $d_k  : \RRR^{n+1}A \to \RRR^n A$ is given by $\RRR^{n-k} \mu_{\RRR^{k-1} A}$, 
  with the convention that
  $\mu_{\RRR^{-1}A} = \alpha$.
  
  We now apply $F\lowershriek$ to the diagram~\eqref{eq:RA} above, to obtain 
  inside $\Grpd$ a
  simplicial-object-with-missing-top-face-maps which we denote $C_\bullet$:
  $$
    \xymatrix@C+2.5em{
    C_\bullet: &
  F\lowershriek A 
  \ar[r]|(0.6){s_0} 
  &
  \ar[l]<+2mm>|(0.6){d_0}\ar@{}[l]<-2.3mm>|{\cdot \  \cdot \ \cdot \ \cdot \ \cdot\ \cdot\ \cdot } 
  F\lowershriek\RRR A 
  \ar[r]<-2mm>|(0.6){s_0}\ar[r]<+2mm>|(0.6){s_1}  
  &
  \ar[l]<+4mm>|(0.6){d_0}\ar[l]|(0.6){d_1}\ar@{}[l]<-4.3mm>|{\cdot \ \cdot \ \cdot \ \cdot \ \cdot\ \cdot\ \cdot }
  F\lowershriek\RRR\RRR A
  \ar[r]<-4mm>|(0.6){s_0}\ar[r]|(0.6){s_1}\ar[r]<+4mm>|(0.6){s_2}  
  &
  \ar[l]<+6mm>|(0.6){d_0}\ar[l]<+2mm>|(0.6){d_1}\ar[l]<-2mm>|(0.6){d_2}\ar@{}[l]<-6.3mm>|{\cdot \  \cdot \ \cdot \ \cdot \ \cdot\ \cdot\ \cdot }
  F\lowershriek\RRR\RRR\RRR A &
  \cdots
  &
  }
  $$
  Finally we apply $\SSS$: the diagram now acquires the missing top face maps,
  and constitutes altogether a simplicial object in $\Grpd$
  (\cite{Weber-CodescCrIntCat}, Lemma 4.3.2) which we denote $\ourX_\bullet$:
  $$
  \xymatrix@C+2.8em{
  \ourX_\bullet : &
  \SSS F\lowershriek A 
  \ar[r]|(0.6){s_0} 
  &
  \ar[l]<+2mm>|(0.6){d_0}\ar@{..>}[l]<-2mm>|(0.6){d_1}
  \SSS F\lowershriek \RRR A 
  \ar[r]<-2mm>|(0.6){s_0}\ar[r]<+2mm>|(0.6){s_1}  
  &
  \ar[l]<+4mm>|(0.6){d_0}\ar[l]|(0.6){d_1}\ar@{..>}[l]<-4mm>|(0.6){d_2}
  \SSS F\lowershriek \RRR\RRR A
  \ar[r]<-4mm>|(0.6){s_0}\ar[r]|(0.6){s_1}\ar[r]<+4mm>|(0.6){s_2}  
  &
  \ar[l]<+6mm>|(0.6){d_0}\ar[l]<+2mm>|(0.6){d_1}\ar[l]<-2mm>|(0.6){d_2}\ar@{..>}[l]<-6mm>|(0.6){d_3}
  \SSS F\lowershriek \RRR\RRR\RRR A &
  \cdots
  }$$
  The new top face maps are given by 
  $$
  d_\top := \mu^\SSS \circ \SSS(\phi) .
  $$
  For example, the first of the top face maps is given as
  $$\xymatrix{
      & \ar[ld]_{\mu^S} \SSS\SSS F\lowershriek A & \\
    \SSS F\lowershriek A && \ar[ll]^{d_1} \ar[lu]_{\SSS\phi}
    \SSS F\lowershriek \RRR A .
  }$$
\end{blanko}

\begin{prop}[\cite{Weber-CodescCrIntCat}, Proposition 4.4.1]
  The simplicial object $\ourX_\bullet$ is a strict category object.
\end{prop}

The statement is that all the squares
$$\xymatrix{
   \ourX_{n+2} \ar[r]^{d_0}\ar[d]_{d_{n+2}} & \ourX_{n+1} \ar[d]^{d_{n+1}} \\
   \ourX_{n+1} \ar[r]_{d_0} & \ourX_{n}
}$$
are strict pullbacks.  As an illustration, the first of these squares ($n=0$)
is the
following, by unravelling the definition of the top face maps:
\begin{equation}\label{phi+mu}
\vcenter{
\xymatrix{
    \SSS F\lowershriek \RRR\RRR A\ar[r]^{\SSS F\lowershriek \RRR \alpha}
    \ar[d]_{\SSS(\phi_{\RRR A})} & \SSS F\lowershriek \RRR A \ar[d]^{\SSS(\phi_A)} \\
   \SSS\SSS F\lowershriek \RRR A \ar[r]^{\SSS\SSS F\lowershriek \alpha}
   \ar[d]_{\mu^\SSS_{F\lowershriek \RRR A}} & \SSS\SSS F\lowershriek A 
   \ar[d]^{\mu^\SSS_{F\lowershriek A}} \\
   \SSS F\lowershriek \RRR A \ar[r]_{\SSS F\lowershriek \alpha} & \SSS 
   F\lowershriek A     .
}}
\end{equation}
The bottom square is a naturality square for the monad multiplication, and is
therefore a strict pullback.  The top square is $\SSS$ applied to a naturality square
for $\phi$.  Since $\phi$ is cartesian and $\SSS$ preserves pullbacks, this
is again cartesian.

\begin{blanko}{Remark.}  
  This simplicial object is a variation of the $2$-sided bar construction,
  which has a long history in algebraic topology~\cite{May:LNM271}.
  The construction here from an operad
  is important in category theory since its codescent object (its lax
  colimit) is the classifier for $\RRR$-algebras (at the present level of 
  generality in symmetric monoidal
  categories), i.e.~the universal symmetric monoidal category containing an internal
  $\RRR$-algebra, cf.~\cite{Batanin-Berger:1305.0086}, \cite{Weber-CodescCrIntCat}. 
  In the present work, we do not take the codescent object, but work
  directly with the simplicial groupoid.
\end{blanko}

\begin{prop}\label{prop:Segal}
  The simplicial groupoid $\ourX_\bullet$ is a Segal groupoid.
\end{prop}
\begin{proof}
  Since we already know that it is a category object, the
  statement is that the strict pullback squares in question are also homotopy
  pullbacks, a typical result for the $2$-categorical approach to 
  polynomial functors.  Since the Segal
  squares are composed of components of $\phi$ and components of $\mu^\SSS$ (as 
  in diagram~\eqref{phi+mu}),
  it is enough to show that $\SSS \circ F\lowershriek$ (the codomain of $\phi$)
  and $\SSS$ itself (the codomain of $\mu^\SSS$) satisfy the conditions of
  the next lemma, which is clear since $I$ is discrete
  and $\B'\to \B$ is a split opfibration.
\end{proof}

\begin{lemma}\label{lem:cart=hocart}
  A strictly cartesian $2$-natural transformation between
  polynomial $2$-functors (between strict slices of $\Grpd$) is also
  homotopy cartesian provided the codomain $2$-functor $\Grpd/I \to \Grpd/J$
  has its lowerstar component along a split 
  opfibration, and the groupoid $I$ is discrete.
\end{lemma}
\begin{proof}
  Proposition~4.6.5 of \cite{Weber:Tbilisi} states that a strictly cartesian
  $2$-natural transformation between polynomial $2$-functors (between strict
  slices of $\Grpd$) is also homotopy cartesian if just the codomain $2$-functor
  is {\em familial} \cite{Weber:TAC18}, which essentially means that it
  preserves fibrations.  On the other hand, familiality is implied by the
  following slight variation of
  Theorem~4.4.5 of \cite{Weber-PolynomialFunctors}, whose notation we use
  freely:  in the proof given in \cite{Weber-PolynomialFunctors}, the condition
  that the lowershriek component is a fibration is not needed. 
  For the required factorisation of $\mathcal{K}/I \to  \mathcal{K}/B$ through 
  $U^{\Phi_{\mathcal{K},B}}$ one only needs
  $s\upperstar$  and $p\lowerstar$ to lift to the level of fibrations, the
  rest of the proof goes as in \cite{Weber-PolynomialFunctors}.
\end{proof}

\begin{prop}\label{prop:D=symmoncatobj}
  The category object $\ourX_\bullet := \SSS F\lowershriek \RRR^\bullet A$ is a symmetric
  monoidal category object in $\Grpd$, and its structure maps $\tensor 
  : \SSS \ourX_\bullet \to \ourX_\bullet$ are
  cartesian and homotopy cartesian on degeneracy maps and inner face maps.
\end{prop}

The symmetric monoidal structure just comes from the fact that 
levelwise $\ourX_\bullet$ is $\SSS$ of something, hence is in fact a {\em free}
$\SSS$-algebra.
The cartesianness is just the cartesianness of the structure maps of the monad 
$\SSS$---this applies to all but the top face maps. Homotopy cartesianness 
follows from Lemma~\ref{lem:cart=hocart}.

\begin{blanko}{The connected Green function.}
  The (weak) slice $\Grpd_{/\ourX_1}$ contains a canonical element, namely
  $$
  G: C_1 \to \ourX_1
  $$
  which we call the {\em connected Green function}.  Recalling
  that $C_1 = F\lowershriek \RRR A$
  and that $\ourX_1 = \SSS F\lowershriek \RRR A$,  the connected Green function is simply
  $\eta^\SSS_{F\lowershriek \RRR A}$, the unit for the monad $\SSS$, the inclusion of
  singleton families into all families. 
  
  We shall also need to split $G$ into its fibres over $\ourX_0$. 
  Denote by $w$ an object in $\ourX_0 = \SSS F\lowershriek A$, a 
  tuple of objects in $F\lowershriek A$.
  Write ${}_w(\ourX_1)$ for the homotopy fibre
  over $w\in \ourX_0$ of the face map $d_1: \ourX_1 \to \ourX_0$
  (reserving the notation
  $(\ourX_1)_w$ for the (homotopy) fibre over $w\in \ourX_0$ of the other face map
  $d_0: \ourX_1 \to \ourX_0$).
  For each $w\in \ourX_0$, consider the (homotopy) pullback squares
  $$\xymatrix{
     {}_w (C_1) \drpullback \ar[r]\ar[d]_{G_w} & C_1 \ar[d]^G \\
      {}_w (\ourX_1) \drpullback \ar[r]\ar[d] & \ourX_1 \ar[d]^{d_1} \\
      1 \ar[r]_{\name{w}} & \ourX_0
  }$$
  so that we have (as in \ref{hosum})
  $$
  G = \int^w G_w .
  $$
  $G_w$ is thus the $w$-ary part of the connected Green function.
\end{blanko}

\begin{lemma}\label{lem:C1C2C1}
  We have (homotopy) pullback squares
  $$\xymatrix{
     C_1 \ar[d] & \drpullback \dlpullback \ar[l]_{d_1} C_2 \ar[d] \ar[r]^{d_0} & C_1 \ar[d] \\
     \ourX_1 & \ar[l]_{d_1} \ourX_2 \drpullback \ar[r]^{d_0} \ar[d]_{d_2}& \ourX_1 \ar[d]^{d_1} \\
     & \ourX_1 \ar[r]_{d_0} & \ourX_0
  }$$
\end{lemma}
\begin{proof}
  The top squares are cartesian since they are naturality squares
  for $\eta: \Id \Rightarrow \SSS$, and homotopy cartesian by 
  Lemma~\ref{lem:cart=hocart}.
  For the bottom square, the assertion follows from \ref{prop:Segal}.
\end{proof}

Stacking the right-hand squares we get the following result, which is
an abstraction of the Key Lemma 5.5 of 
\cite{GalvezCarrillo-Kock-Tonks:1207.6404}.

\begin{cor}\label{cor:C2}
  We have a canonical equivalence of groupoids
  $$
  C_2 \simeq \ourX_1 \times_{\ourX_0} C_1.
  $$
\end{cor}

This pullback we now split into its fibres over $\ourX_0$ as 
in Lemma~\ref{lem:XSY}:

\begin{prop}\label{prop:equiv}
  We have a canonical equivalence
  $$
  C_2 \simeq \int^w (\ourX_1)_w \times {}_w(C_1) .
  $$
  This equivalence is over $\ourX_1 \times C_1$ and hence also over
  $\ourX_1 \times \ourX_1$.
\end{prop}

This is the essence of the Fa\`a di Bruno formula, as we proceed to explain in the following
sections.  The groupoid $C_2$ (with its map to $\ourX_1 \times \ourX_1$) is the
left-hand side of the Fa\`a di Bruno formula, and $\int^w (\ourX_1)_w \times
{}_w(C_1)$ (with its map to $\ourX_1 \times \ourX_1$) is the right-hand side.

We now analyse the groupoid $(\ourX_1)_w$, for $w\in \ourX_0 = \SSS C_0$. 
Recall that $\SSS$ is defined by the polynomial $1 \leftarrow \B' \to \B \to 1$.
For $n\in \B$, let 
$\underline n$ denote the fibre over $n$ of the projection $\B'\to\B$.
With this notation, the formula for evaluation of $\SSS$ reads
$\SSS(X) = \int^{n\in \B} \Map(\underline n,X)$.
The element $w \in \ourX_0 = \SSS C_0 = \int^{n\in \B} 
\Map(\underline n,C_0)$ thus amounts to a map $w: \underline n \to C_0$ (for 
some $n$).

\begin{lemma}\label{lem:C_1^w}
  For fixed $w : \underline n \to C_0$, we have a canonical equivalence
  $$
  (\ourX_1)_w \simeq \Map_{/C_0}(\underline n, C_1)  =: C_1^w.
  $$
  Here $\underline n$ is considered over $C_0$ via $w$, and
  $C_1$ is considered over $C_0$ via $d_0$.
\end{lemma}
\begin{proof}
  The map $d_0 : \ourX_1 \to \ourX_0$ is $\SSS$ applied to 
  $d_0 : C_1 \to C_0$, and hence, by 
  definition of $\SSS$, can be written
  $$
  \int^{n\in \B} \big( \Map(\underline n,C_1) \to \Map(\underline n,C_0) \big)
  $$
  which for fixed $n$ is just post-composition with $d_0: C_1 \to C_0$.
  The $w$-fibre of $\ourX_1 \to \ourX_0$ is therefore computed by the standard 
  slice-mapping-space fibre sequence (see~\cite{Carlier-Kock}):
  \[
  \begin{gathered}
    \xymatrix{\Map_{/C_0}(w,d_0) \drpullback \ar[r]\ar[d] & \Map(\underline n,C_1)
      \ar[d]^{\text{post } d_0} \\
      1 \ar[r]_-{\name{w}} & \Map(\underline n,C_0) .}
  \end{gathered}
  \qedhere
  \]
\end{proof}

This lemma is an abstraction of Lemma~5.2 in 
\cite{GalvezCarrillo-Kock-Tonks:1207.6404}.  It can be 
interpreted as saying that an object of $(\ourX_1)_w$ is an $n$-tuple
of elements in $C_1$, with specified output colours given by the $n$-tuple
$w: \underline n \to C_0$.
In the lemma we have already introduced the notation $C_1^w$ to reflect this 
interpretation.

The interpretation so far is phrased abstractly 
in terms of $C_i$, already with a view towards the generalisations in
Section~\ref{sec:generalised}.  In fact we can unpack further, to get down
to groupoids of operations of the operad $\RRR$.  As a first step, we
unravel the square in Corollary~\ref{cor:C2}:

\begin{lemma}
  The pullback square
  $$\xymatrix{
     C_2\drpullback \ar[r]^{d_0}\ar[d]_{d_2 \circ \eta} & C_1 \ar[d]^{d_2 \circ \eta} \\
     D_1 \ar[r]_{d_0} & D_0
  }$$
  is naturally identified with the naturality square for $\phi$,
  \begin{equation}\label{eq:basicNatSquareForFdB}
    \vcenter{
    \xymatrix{
     F\lowershriek \RRR\RRR A \drpullback \ar[r]^{ F\lowershriek \RRR \alpha}
    \ar[d]_{\phi_{\RRR A}} &  F\lowershriek \RRR A \ar[d]^{\phi_A} \\
   \SSS F\lowershriek \RRR A \ar[r]_{\SSS F\lowershriek \alpha}
   & \SSS F\lowershriek A  .
  }}
  \end{equation}
\end{lemma}
\begin{proof}
  The square is composed of two squares as in Lemma~\ref{lem:C1C2C1}, the bottom being 
  in turn composed of the two squares in \eqref{phi+mu}.  Altogether three 
  squares
  are stacked, with sides $\eta$, then $\SSS\phi$, then $\mu$.
  Now $\eta$ and $\SSS\phi$ can be interchanged by naturality, and then $\eta$ and 
  $\mu$ cancel out, by the unit law of $\SSS$, leaving only
  the asserted square~\eqref{eq:basicNatSquareForFdB}, naturality for $\phi$.
\end{proof}

Now we unpack further.
First note that $F\lowershriek A =I$ and $F\lowershriek \RRR A = B$, with 
reference to the polynomial $I \leftarrow E \to B \to I$ representing $\RRR$.
The effect on objects of $\phi_A : B \to \SSS I$ is to send an operation 
$b : (i_1,...,i_n) \to i$ to the input sequence $(i_1,...,i_n)$.
The effect on objects of $\SSS F\lowershriek \alpha : \SSS B \to \SSS I$ 
is to send a sequence of operations $(b_1,...,b_n)$ to the
sequence of output colours of the $b_j$. For an arbitrary object
$w = (i_1,...,i_n)$ of $\SSS I$, the homotopy fibres ${}_wB$ 
and $(\SSS B)_w$ of these functors are easily computed. 
The groupoid ${}_wB$ has
\begin{itemize}
  \item objects triples $(\rho,b,j)$ where $\rho \in \Sigma_n$ and
  $b : (i_{\rho 1},...,i_{\rho n}) \to j$ is an operation of $\RRR$.
  \item arrows $(\rho,b,j) \to (\rho',b',j')$ are $\psi \in \Sigma_n$
  such that $\rho = \rho'\psi$ and $b = b'\psi$.
\end{itemize}
Thus ${}_wB$ is a groupoid of operations with input sequence 
some permutation of $(i_1,...,i_n)$. The groupoid $(\SSS B)_w$ has
\begin{itemize}
  \item objects pairs $(\rho,(b_1,...,b_n))$ where $\rho \in \Sigma_n$ and
  $b_j \in B$ such that $tb_j = i_{\rho j}$ for $1 \leq j \leq n$.
  \item arrows $(\rho,(b_1,...,b_n)) \to (\rho',(b'_1,...,b'_n))$ are
  permutations $\psi \in \Sigma_n$ and $(\psi_1,...,\psi_n)$, such that 
  $\rho = \rho'\psi$ and $b_j = b'_{\psi(j)}\psi_j$.
\end{itemize}
Thus $(\SSS B)_w$ is a groupoid of $n$-tuples of operations 
with bijection associating output colours with $(i_1,...,i_n)$. 
As in Lemma~\ref{lem:C_1^w}, it is natural to denote $(\SSS B)_w$ as $B^w$.
The groupoid $C_2 = F\lowershriek\RRR\RRR A$ is the groupoid 
of operations of $\RRR$ labelled by operations of $\RRR$, 
described explicitly in the proof of Proposition~3.6 of \cite{Weber-OpPoly2Mnd}.
In pictorial terms, an object here 
is depicted as
\[ \xygraph{!{0;(.9,0):(0,.9)::} 
{\scriptstyle{b}} *\xycircle<6pt,6pt>{-}="p0" [ul]
{\scriptstyle{b_1}} *\xycircle<6pt,6pt>{-}="p1" [r(2)]
{\scriptstyle{b_k}} *\xycircle<6pt,6pt>{-}="p2"
"p0" (-"p1",-"p2",-[d],[u(.75)] {...},[u(.5)l(.7)] {\scriptstyle{i_1}},[u(.5)r(.75)] {\scriptstyle{i_k}},[d(.7)r(.15)] {\scriptstyle{i}})
"p1" [u(1)l(.5)] {}="q1" [r] {}="q2"
"p1" (-"q1",-"q2",[u(.75)] {...},[u(.45)l(.47)] {\scriptstyle{i_{11}}},[u(.45)r(.55)] {\scriptstyle{i_{1n_1}}})
"p2" [u(1)l(.5)] {}="r1" [r] {}="r2"
"p2" (-"r1",-"r2",[u(.75)] {...},[u(.45)l(.47)] {\scriptstyle{i_{k1}}},[u(.43)r(.55)] {\scriptstyle{i_{kn_k}}})} \]
Thus Proposition~\ref{prop:equiv} unpacks to

\begin{cor}\label{cor:B2RasHomSum}
For any operad $\RRR$, with notation as above,
\[ C_2 \simeq \int^{w \in \SSS I} B^w \times {}_wB \]
over $\SSS B \times B$.
\end{cor}
\noindent
in clear analogy with the classical Fa\`a di Bruno formula, especially in the 
uncoloured case, $I=1$, where it reduces to
$C_2 \simeq \int^{n} B^n \times {}_n B$.

\section{Monoidal decomposition spaces and bialgebras}

\label{sec:decomp}

We now proceed to explain the bialgebra structure (at the groupoid level) in 
which to interpret the equivalence established above.  The actual work was done 
above; we just need to apply the results of 
G\'alvez--Kock--Tonks~\cite{Galvez-Kock-Tonks:1512.07573}, and our task is only to 
explain how it works.

\begin{blanko}{Decomposition spaces and coalgebras.}
  The notion of decomposition space was introduced in 
  \cite{Galvez-Kock-Tonks:1512.07573}, as a generalisation of posets and 
  M\"obius categories for the purpose of defining incidence coalgebras.
  A decomposition space is a
  simplicial groupoid $X_\bullet: \simplexcategory\op\to\Grpd$ (in full generality 
  a simplicial $\infty$-groupoid)
  satisfying an exactness condition ensuring that the canonical span
  \begin{equation}\label{eq:span}
  X_1 \stackrel{d_1}\longleftarrow X_2 \stackrel{(d_2,d_0)}\longrightarrow X_1 
  \times X_1  
  \end{equation}
  induces a homotopy-coherently coassociative coalgebra structure on
  $\Grpd_{/X_1}$ in a sense we shall now detail.  The precise
  exactness condition is not necessary in the present work: all we need
  to know is that Segal spaces are decomposition spaces
  \cite[Proposition 3.5]{Galvez-Kock-Tonks:1512.07573}. 
  
  The idea behind the coalgebra construction is simple,
  and goes back to Leroux~\cite{Leroux:1975}: an arrow $f$ of a category
  object is comultiplied by the formula
  $$
  \Delta(f) = \sum_{b\circ a=f} a \tensor b .
  $$
  The sum is over all ways of factoring $f$ into two arrows, generalising the
  way intervals are comultiplied in the classical incidence coalgebra of a
  poset~\cite{Joni-Rota}.  The above span~\eqref{eq:span} defines a linear functor
  $$
  \Grpd_{/X_1} \stackrel{(d_2,d_0)\lowershriek \circ 
  d_1\upperstar}{\longrightarrow} \Grpd_{/X_1 \times X_1}
  $$
  which is precisely the objective version of $\Delta$: pullback to $X_2$ means
  taking all triangles with long edge $f$, and then the lowershriek means 
  returning the two short edges.  ({\em Linear functor} means that it preserves 
  homotopy sums; the linear functors are precisely those given by pull-push along
  spans~\cite{Galvez-Kock-Tonks:1602.05082}.)
  
  Similarly, the span
$$
X_1 \stackrel{s_0}{\longleftarrow} X_0 \stackrel{u}{\longrightarrow} 1 
$$
defines the linear functor
$$
  \Grpd_{/X_1} \stackrel{u\lowershriek \circ s_0\upperstar}\longrightarrow  
  \Grpd 
$$
which is the counit for $\Delta$.
 
  Altogether, $\Grpd_{/X_1}$ becomes a coalgebra object in the symmetric 
  monoidal $2$-category  $(\LIN, \tensor, \Grpd)$ \cite{Galvez-Kock-Tonks:1602.05082},
  whose objects are (weak) slices and whose
  morphisms are linear functors.  The
  monoidal structure $\tensor$ is given by the formula
  $$
  \Grpd_{/A} \tensor \Grpd_{/B} = \Grpd_{/A\times B}
  $$
  with neutral object $\Grpd_{/1}$, and the symmetry is induced by the canonical 
  symmetry of the cartesian product.
\end{blanko}

\begin{blanko}{Monoidal decomposition spaces.}
  In general, a sufficient condition for a simplicial map $f: Y_\bullet \to 
  X_\bullet$ between
  decomposition spaces to induce a coalgebra homomorphism on incidence
  coalgebras, is that $f$ be {\em CULF}~\cite{Galvez-Kock-Tonks:1512.07573},
  which stands for {\em conservative}
  (meaning that it does not invert any arrows) and {\em unique lifting of
  factorisations} (meaning that for an arrow $a \in  Y_1$, there is a one-to-one
  correspondence between the factorisations of $a$ in $Y_\bullet$ and the factorisations
  of $f(a)$ in $X_\bullet$).  A simplicial map is CULF precisely when it is 
  homotopy cartesian on
  degeneracy maps and inner face maps \cite{Galvez-Kock-Tonks:1512.07573}.
  
  Recall that a bialgebra is a monoid object in the category of coalgebras,
  meaning that multiplication and unit are homomorphisms of coalgebras.
  Accordingly, in order to induce a bialgebra structure on $\Grpd_{/X_1}$, we
  need on $X_\bullet$ a monoidal structure which is CULF.
  This motivates defining a {\em monoidal decomposition space}
  \cite{Galvez-Kock-Tonks:1512.07573} to be a decomposition space $X_\bullet$ equipped with
  a monoidal structure $ \ \tensor : \SSS X_\bullet \to X_\bullet$, whose structure maps are
  CULF. 
\end{blanko}

\begin{blanko}{Bialgebra structure on $\Grpd_{/\ourX_1}$.}
  Coming back now to the simplicial groupoid $\ourX_\bullet = \SSS F\lowershriek
  \RRR^\bullet A$, we have precisely such a symmetric monoidal structure, namely
  given by the multiplication map $\mu^\SSS : \SSS(\SSS F\lowershriek
  \RRR^\bullet A) \to \SSS F\lowershriek \RRR^\bullet A$, which is CULF 
  by Proposition~\ref{prop:D=symmoncatobj}.
  In conclusion, by Theorem~7.3 and Proposition~9.5 of \cite{Galvez-Kock-Tonks:1512.07573},
  we have:
\end{blanko}

\begin{prop}
  The slice category $\Grpd_{/\ourX_1}$ is a bialgebra object in 
  $(\LIN,\tensor,\Grpd)$.
\end{prop}
\noindent
With the current choice of $\SSS$ it is even a `symmetric' bialgebra, meaning 
that the homotopy cardinality will be a commutative bialgebra, as we shall see 
in the next section.

We can now give the bialgebra reformulation of the equivalence in
Proposition~\ref{prop:equiv}, which is the Fa\`a di Bruno formula at
the objective level:

\begin{blanko}{Theorem.}
  \em
  We have the following equivalence in $\Grpd_{/\ourX_1\times \ourX_1}$:
  $$
  \Delta(G) = \int^w G^w \times  G_w .
  $$
\end{blanko}

\section{Finiteness conditions and homotopy cardinality}

\label{sec:card}

When working at the objective level of groupoid slices, the results so far hold
for {\em any} operad $\RRR$.  However, in order to take cardinality, it is
necessary to subject $\RRR$ to certain finiteness
conditions~\cite{Galvez-Kock-Tonks:1512.07577}.

In this section we explain the procedure of taking homotopy cardinality, following
\cite{Galvez-Kock-Tonks:1602.05082}.  In that paper, the setting is that
of $\infty$-groupoids, but everything works also for
$1$-groupoids~\cite{Carlier-Kock}. 

\begin{blanko}{Homotopy cardinality of groupoids and slices~\cite{Galvez-Kock-Tonks:1602.05082}.}
  A groupoid $X$ is called {\em homotopy finite} when it has only finitely many
  connected components, and all its automorphism groups are finite.  In that case,
  the homotopy cardinality of $X$ is defined as
  $$
  \norm{X} := \sum_{x\in \pi_0 X} \frac{1}{\norm{\Aut{x}}} .
  $$
  Let $\grpd$ (lowercase) denote the category of finite groupoids.

  Groupoids of combinatorial objects are {\em not} usually finite, because they
  have infinitely many components, but they usually have finite 
  automorphism groups.  Such groupoids are called {\em locally 
  finite}~\cite{Galvez-Kock-Tonks:1602.05082}.  From now on we assume that
  all groupoids are locally finite.
  For $X$ a locally finite groupoid, the basic vector space is $\Q_{\pi_0 X}$,
  the free vector space on the set of iso-classes of objects in $X$.  We denote 
  the basis elements $\delta_x$ (for $x\in \pi_0 X$). 
  Since linear combinations in a vector space are finite by definition, the 
  correct groupoid slice to consider is $\grpd_{/X}$ of finite groupoids over 
  $X$.  The cardinality of an element $A \to X$ in $\grpd_{/X}$ is the vector
  $$
  \sum_{x\in\pi_0 X} \frac{\norm{A_x} }{\norm{\Aut x}} \; \delta_x 
  \quad  \in \Q_{\pi_0 X} .
  $$
  Linear maps $\Q_{\pi_0 X} \to \Q_{\pi_0 Y}$ are modelled by linear functors 
  $\grpd_{/X} \to \grpd_{/Y}$, in turn given by spans of {\em finite type}
  $$
  X \stackrel p\leftarrow M \stackrel q\to Y ,
  $$
  meaning that $p$ is a finite map (i.e.~has homotopy-finite homotopy fibres).
  
  The notion of homotopy cardinality has all the expected properties: it takes 
  equivalences to equalities, it preserves products, sums, quotients---and hence 
  homotopy sums; it takes monoidal structures to monoids.
\end{blanko}

\begin{blanko}{Locally finite decomposition spaces.}
  For $X_\bullet$ a decomposition space, we consider the finite-groupoid slice
  $\grpd_{/X_1} \subset\Grpd_{/X_1}$, which is well behaved assuming $X_1$ is 
  locally finite.
  For the
  coalgebra structure maps $\Delta: \Grpd_{/X_1} \to \Grpd_{/X_1\times X_1}$ and
  $\varepsilon: \Grpd_{/X_1} \to \Grpd$ to descend to the finite slice
  $\grpd_{/X_1}$ it is therefore sufficient to require that
  $d_1 : X_2 \to X_1$ and $s_0: X_0 \to
  X_1$ be finite maps.  Decomposition spaces with this property (and $X_1$ 
  locally finite) are called {\em
  locally finite} \cite[\S7]{Galvez-Kock-Tonks:1512.07577} (extending the notion
  for posets \cite{Joni-Rota}).  (It may be noted that there are no conditions
  on the algebra structure: it always descends to finite-groupoid slices, since
  it is a pure lowershriek operation.)
\end{blanko}

\begin{blanko}{Locally finite operads (and monads).}
  Coming back to operads, we can ensure the local finiteness condition on
  $\ourX_\bullet$ by the following requirement.  Call a monad $\RRR$ {\em
  locally finite} when $\mu: \RRR\RRR \Rightarrow \RRR$ and $\eta:
  \Id\Rightarrow \RRR$ are finite natural transformations (i.e.~all components
  are finite maps). 
  What it amounts to is that for every operation $r$, there
  are only finitely many ways of writing it
  $$
  r = b \circ (a_1,\ldots,a_n)
  $$
  for $b$ an $n$-ary operation (and $a_i$ operations whose arities add up to 
  that of $r$).
  
  As an non-example, the commutative monoid monad $\SSS$ is
  {\em not} locally finite, because the identity operation $u \in \SSS_1$ could be
  obtained in infinitely many ways as a composition: by filling $n-1$ nullary 
  operations into an $n$-ary operation (for all $n$).  In contrast, 
  cf.~Example~\ref{ex:clas} below, the 
  commutative semimonoid monad $\SSS_+$ {\em is} locally finite, since there 
  are no nullary operations to screw things up.  Let us observe that the 
  existence of nullary operations is not formally an obstruction to being 
  locally finite, as long as they are not subject to relations.  For example,
  every free monad (on a polynomial endofunctor) is locally finite, 
  cf.~Example~\ref{ex:free} below.
  
  The local finiteness of $\RRR$ makes makes $C_\bullet$ locally finite (it is not quite
  a simplicial object, because the top face maps are missing,
  but its does feature the face and degeneracy maps on 
  which the condition is measured).  Since $F_!$ and $\SSS$ preserve finite maps,
  $\ourX_\bullet$ is also locally finite.
\end{blanko}
  
\begin{blanko}{The connected Green function as a homotopy cardinality.}
  The connected Green function lives in the completion of $\Q_{\pi_0 \ourX_1}$,
  where we allow infinite sums (but not infinite coefficients).  This
  (profinite-dimensional) vector space arises as the homotopy cardinality of the
  bigger slice $\Grpd_{/\ourX_1}^\relfin$ whose objects are finite maps $E \to 
  \ourX_1$
  (but $E$ itself not required finite).  The relevant notion of linear map
  (continuous in a certain `pro' sense) are given by spans
    $$
  X \stackrel p\leftarrow M \stackrel q\to Y
  $$
  where instead
  the right leg $q$ is finite~\cite{Galvez-Kock-Tonks:1602.05082}.
  
  In order to formulate the Main Theorem at the vector space level,
  we have to justify that the
  comultiplication extends to this bigger slice, which is to say we must check
  that $\ourX_1 \times_{\ourX_0} \ourX_1 \to \ourX_1 \times \ourX_1$ is a finite map.  This is a
  standard argument~\cite{Galvez-Kock-Tonks:1602.05082}: this map sits 
  naturally in a homotopy pullback square
  $$\xymatrix{
     \ourX_1 \times_{\ourX_0} \ourX_1 \drpullback \ar[r]\ar[d] & \ourX_0 \ar[d]^{\text{diag.}} \\
     \ourX_1 \times \ourX_1 \ar[r] & \ourX_0\times \ourX_0
  }$$
  so it is enough to show that the diagonal $\ourX_0 \to \ourX_0 \times \ourX_0$
  is finite.  It is automatically discrete (i.e.~$0$-truncated) since $\ourX_0$
  is $1$-truncated, and the homotopy fibre at a point $(w,w)$ is essentially the
  set of automorphisms of $w$, which is finite (for all $w$) if and only if
  $\ourX_0$ is locally finite.  Since we have already assumed that $\ourX_1$ is
  locally finite, and that $s_0 : \ourX_0 \to \ourX_1$ is finite, it follows
  that $\ourX_0$ is again locally finite.  Therefore, the comultiplication
  extends to the bigger slice $\Grpd_{/\ourX_1}^\relfin$ as required.  It may be
  noted that the counit does {\em not} extend to $\Grpd_{/\ourX_1}^\relfin$.
  This would require $\ourX_0\to 1$ to be a finite map, which is never the case
  since it is $\SSS$ of something.  This means that the coalgebra
  $\Grpd_{/\ourX_1}^\relfin$ is not counital.  This is not really an issue: our
  main interest is the counital coalgebra $\grpd_{/\ourX_1}$.  The extension to
  $\Grpd_{/\ourX_1}^\relfin$ is introduced only to be able to state the Fa\`a di
  Bruno formula efficiently in terms of the connected Green function---it does
  not involve the counit at all.
\end{blanko}

\begin{blanko}{Cardinality of the main equivalence.}
  We assume now that $\RRR$ is locally finite, and proceed to take homotopy
  cardinality of the main equivalence.  For clarity, we temporarily use 
  underline notation for the symbols at the algebra level.
  By construction, the family $G: C_1 \to \ourX_1$ has cardinality $\underline
  G$, and the linear functor $\Delta$ has cardinality $\underline \Delta$.
  Hence the left-hand side of the equation is clear: the cardinality of $C_2$ is
  $\underline\Delta(\underline G)$.
  
  For the right-hand side, note first that cardinality preserves homotopy sums
  (i.e.~transforms the integral into a sum over $\pi_0$ divided by symmetry 
  factors).  Now $(\ourX_1)_w \simeq C_1^w$ (as an object over $\ourX_1$) has
  cardinality $\underline G^w$, and $G_w$ has cardinality $\underline G_w$.
%
%
%
  In conclusion, homotopy cardinality of the main groupoid equivalence
  yields the Fa\`a di Bruno formula at the algebraic level (where we no longer 
  write underlines):
\end{blanko}

\begin{blanko}{Theorem.}\label{thm:main} \em
  In the incidence bialgebra of a locally finite operad $\RRR$ we have
  $$
  \Delta(G) = \sum_w G^w \tensor G_w/\norm{\Aut(w)} .
  $$   
\end{blanko}

\section{Generalisations}

\label{sec:generalised}

So far we have treated the case of the terminal algebra for an operad.
Several generalisations follow readily by closer inspection of the 
constructions and proofs, as we proceed to explain.

\begin{blanko}{Non-discrete colours.}
  In the polynomial monads corresponding to operads, the groupoid of colours $I$
  is discrete.  We wish to give up that condition, because there are interesting
  examples with non-discrete $I$, namely operads in groupoids
  (cf.~\ref{ex:operadmonad} and \ref{ex:QFT} below).  The only place where the
  discreteness condition was used, was in the proof of
  Proposition~\ref{prop:Segal}, where it ensured that the strictly cartesian
  $2$-natural transformation $\phi: F\lowershriek \RRR \Rightarrow \SSS
  F\lowershriek$ is also homotopy cartesian.  But in fact we needed this only
  for naturality squares on arrows of the type $F\lowershriek \alpha$ (and
  $\SSS$ applied to those).  If instead we require $F\lowershriek \alpha$ to be
  a fibration, we can draw the same conclusion and establish
  Proposition~\ref{prop:Segal} again, without discreteness assumptions.
  $F\lowershriek \alpha$ is precisely
  the map of groupoids $B \to I$ in the diagram defining $\RRR$, so we impose
  now the condition that this is a fibration.  (Note that it is always a
  fibration when $I$ is discrete.)
\end{blanko}

\begin{blanko}{Algebras.}\label{algebras}
  The second generalisation is to note that to set up the Segal space 
  $\ourX_\bullet$, we did not rely on the algebra $\alpha: \RRR A \to A$
  being the terminal algebra.  In fact, all the arguments in the construction
  work exactly the same for any (strict) $\RRR$-algebra.  We should still
  require, though, that $F\lowershriek \alpha$ is a fibration of groupoids
  (or that $I$ is discrete).
\end{blanko}

\begin{blanko}{General $\SSS$.}
  So far we have assumed that $\SSS$ is the symmetric monoidal category
  monad, so that monads $\RRR$ over it correspond to operads (and operads in
  groupoids).  In fact, all the results generalise readily to any general finitary
  polynomial monad $\SSS$ (although for simplicity, we assume that the
  maps in the representing polynomial diagram are fibrations):
  the general situation concerns any cartesian monad morphism between polynomial
  monads
  $$
  \RRR\stackrel F \Rightarrow \SSS
  $$
  and leads to $\RRR$-algebras internal to categorical $\SSS$-algebras (as developed by
  \cite{Baez-Dolan:9702, Batanin:0207281, Batanin-Berger:1305.0086, Weber-OperadsInMonPsAlgebras, Weber-CodescCrIntCat}).
  The fibrancy condition on $(A,\alpha)$ is still required, of course.

  One thing that changes when $\SSS$ is no longer the symmetric monoidal category monad
  is that the category object $\ourX_\bullet$ is no longer symmetric monoidal but is instead
  a categorical $\SSS$-algebra, now for the general $\SSS$. 
  Each choice of $\SSS$ will give a new ambient setting, and a new notion of 
  connectedness: the connected Green function will always be given by the unit of 
  the monad.
 
  \bigskip
  
  Since the proofs in Section~\ref{sec:operads} did not actually use other 
  properties of the monad $\SSS$ than being polynomial, cartesian and homotopy cartesian,
  we get immediately in this more general situation:
\end{blanko}

\begin{prop}
  $\Grpd_{/\ourX_1}$ is an $\SSS$-algebra in $\kat{Coalg}(\LIN)$.
\end{prop}

Accordingly, taking 
cardinality will yield an $\SSS$-algebra object in the category of 
coalgebras, or, equivalently, a comonoid in the category of $\SSS$-algebras.

In the operad case, the setting for the Fa\`a di Bruno formula is a free 
algebra (or more precisely a power-series ring), and it is straightforward
to interpret the exponent $w$ in the left-hand factor $G^w$: it is simply
$G$ multiplied with itself $w$ times, where in reality $w$ represents the
shape of a list of things to be multiplied.
For general $\SSS$, the exponent $w$ represents the shape of an $\SSS$-configuration,
and $G^w$
denotes such a configuration of elements in $G$, which can be 
multiplied using the (free) $\SSS$-algebra structure.  The precise meaning of 
the exponent is given in \ref{lem:C_1^w}, and hopefully the examples 
will clarify this point.

\bigskip

The most obvious alternative to the symmetric monoidal category monad
is to take $\SSS$ to be the (nonsymmetric) monoidal category monad, which amounts
to considering nonsymmetric operads.  The outcome are noncommutative 
bialgebras,
and some kind of noncommutative Fa\`a di Bruno formula.  See Example~\ref{ex:LN}
below.

\section{Examples}

\label{sec:ex}

In the first few examples we maintain as  
$\SSS$ the symmetric monoidal category monad.
Hence the bialgebras will be free commutative.
We also keep a discrete groupoid of colours.

\begin{blanko}{Classical Fa\`a di Bruno.}\label{ex:clas}
  Take $\RRR$ to be $\SSS_+$, the polynomial monad for commutative semimonoids.
  It is like $\SSS$, but omitting the nullary operation.  It is the terminal
  reduced operad $\mathsf{Comm}_+$ (the word `positive' is also used instead of
  `reduced'~\cite{Aguiar-Mahajan}).  One can check that the bar category
  $\ourX_\bullet$ is then the fat nerve (see 
  \cite[2.14]{Galvez-Kock-Tonks:1512.07573}) of the category 
  of finite sets and 
  surjections. Indeed, $\ourX_0$ is $\SSS 1$, which is equivalent to 
  the groupoid of finite sets and bijections---this is also degree $0$ of
  the fat nerve. Next, $\ourX_1$ is $\SSS\SSS_+ 1$,
  the groupoid of symmetric lists of nonempty symmetric lists of $1$s, 
  naturally equivalent with the groupoid of surjections, degree $1$ of the fat 
  nerve.  The same identifications work for general $n$.  This fat nerve is
  symmetric monoidal under disjoint union, and the resulting bialgebra is the classical
  Fa\`a di Bruno bialgebra (this observation goes back to \cite{Joyal:1981}; see
  \cite{GKT:ex} for an elaboration), and the resulting Fa\`a di Bruno formula is
  the classical one.
\end{blanko}

\begin{blanko}{Multivariate Fa\`a di Bruno.}\label{ex:multivar}
  We consider the multi-variate version of the previous example.  Let $I$ be a
  set of colours.  Let $\RRR$ be the operad whose $n$-ary operations are
  $(n+1)$-tuples of colours (and without nullary operations).  The symmetries
  are colour-preserving bijections respecting the base point.  One may think of
  these operations as corollas whose leaves and root are decorated with elements
  in $I$.  The substitution operation takes a two-level trees with decorated
  edges and contracts the inner edges, simply forgetting their colours.  Note
  that since the inner-edge colours are simply lost in this process, many
  substitutions give the same result.  For this reason, in order for $\RRR$ to be
  locally finite we must demand $I$ to be a finite set.
  
  The bar construction of $\RRR$ is the category object in groupoids with $\ourX_0$
  the free groupoid on $I$, that is, tuples of colours and colour-preserving
  bijections.  $\ourX_1$ is the groupoid whose elements are coloured surjections,
  and whose arrows are pairs of colour-preserving bijections. We point out that
  this Segal groupoid is not Rezk
  complete (see \cite{Galvez-Kock-Tonks:1512.07577} for discussion of this 
  issue): for any two colours $i$ and $j$ there is a unary operation $i \to
  j$, with inverse $j \to i$.  These invertible unaries do not come from 
  $\ourX_0$,
  which only contains colour preserving bijections.
  
  The bialgebra resulting from $\ourX_\bullet$ is the polynomial algebra 
  generated by symbols $A_{w,i}$ with $i\in I$ and $w$ a nonempty word in $I$, one 
  generator for each iso-class of connected coloured surjections.  Closely 
  related to the non-Rezk-ness of $\ourX_\bullet$ is the fact that the identity
  surjections are not group-like.  Indeed, we have
  $$
  \Delta(A_{i,i}) = \sum_{j\in I} A_{i,j} \tensor A_{j,i}
  $$
  expressing that the identity $i \to i$ admits factorisations $i \to j \to i$.
  This bialgebra is precisely the dual to composition of multi-variate power 
  series; more precisely $I$-tuples of power series in $I$-many 
  variables.  (The linear (unary) part of this substitution is simply matrix 
  multiplication, and we recognise indeed the above formula as dual to the 
  formula for the $i$th diagonal entry in a matrix product.)
\end{blanko}

\begin{blanko}{Algebras.}
  Continuing the case $\RRR=\SSS_+$,
  let $A$ be an $\SSS_+$-algebra, i.e.~a commutative semimonoidal groupoid.
  As in the two previous 
  examples, the resulting bar construction $\ourX_\bullet$ is a certain groupoid-enriched 
  category of decorated surjections, but this time the decorations on each corolla are
  by elements in $A$ such that the output colour is the semimonoid product of
  all the input colours.  In particular, there are no unary operations $i \to j$
  for $i\neq j$, and it follows readily that $\ourX_\bullet$ is Rezk complete
  this time, just as in \ref{ex:clas}.
  
  The coalgebra structure does not change much from that in
  \ref{ex:clas}, because only the underlying surjection really matters: the
  possible decorations in the factorisations are completely determined by the
  decorations of the original surjection.
\end{blanko}

In the next three examples, we give up the requirement that $I$ is discrete.

\begin{blanko}{Free operads (Fa\`a di Bruno for 
  trees~\cite{GalvezCarrillo-Kock-Tonks:1207.6404}).}\label{ex:free}
  Let $P$ be any finitary polynomial endofunctor, and let $\RRR$ denote the 
  free monad  on $P$.  Its operations are the $P$-trees.  Assuming that the 
  groupoid $P1$ is 
  locally finite, also the groupoid of $P$-trees is locally finite, and the 
  factorisations of operations amount to cuts in trees (as in 
  \cite{Kock:1109.5785}), and since a given tree admits only finitely many 
  cuts, it follows that $\RRR$ is locally finite.
  The resulting bialgebra is the
  $P$-tree version (\cite{Kock:1109.5785}) of the Butcher--Connes--Kreimer Hopf
  algebra~\cite{Connes-Kreimer:9808042} (but note that it is a bialgebra not a
  Hopf algebra: it fails to be connected for the node grading,
  because all the nodeless forests are of
  degree $0$).  The corresponding Fa\`a di Bruno formula is that from
  \cite{GalvezCarrillo-Kock-Tonks:1207.6404}.  
  
  The {\em core} of a $P$-tree is the combinatorial tree obtained by forgetting
  the $P$-decoration and shaving off leaves and root~\cite{Kock:1109.5785}.
  Taking core defines a bialgebra homomorphism from the bialgebra of $P$-trees
  to the usual Butcher--Connes--Kreimer Hopf algebra, thereby producing Fa\`a
  di Bruno sub-Hopf algebras, including virtually all the ones of Bergbauer
  and Kreimer~\cite{Bergbauer-Kreimer:0506190} (but not those of
  Foissy~\cite{Foissy:0707.1204}, whose grading is of a different nature); (see
  \cite{Kock:1512.03027} for these results).
\end{blanko}

In the next example, it is interesting also to consider arbitrary algebras.

\begin{blanko}{Node substitution in trees: the monad for operads.}\label{ex:operadmonad}
  Symmetric operads are themselves algebras for a polynomial monad $\RRR$; it
  is given by
  $$
  \B \leftarrow \T\upperstar \to \T \to \B
  $$
  where (as usual) $\B$ is the groupoid of finite sets and bijections, 
  $\T$ is groupoid of rooted (operadic) trees, and $\T\upperstar $ is
  the groupoid of rooted (operadic) trees with a marked node.
  The map $\B \leftarrow \T\upperstar $ returns the set of incoming
  edges of the marked node, $\T\upperstar \to\T$ forgets the mark, and
  $\T\to\B$ returns the set of leaves.  An $\RRR$-algebra is thus a
  map $A \to \B$ (which should be required to be a discrete fibration
  in order to get the usual notion of $\Set$-operad), whose (homotopy)
  fibre over $n \in \B$ is thought of as the set of $n$-ary operations.
  $\RRR A$ is the groupoid of trees whose nodes are decorated by
  operations in $A$, and the monad structure $\alpha: \RRR A \to A$
  gives precisely the operad substitution law, prescribing how to
  contract a whole tree configuration of operations to a single operation.
  
  To ensure local finiteness of $\RRR$, one should exclude the trivial
  tree, and the resulting notion of operad is then that of non-unital operad.
  
  The incidence bialgebra of $\RRR$, corresponding to the terminal operad,
  is the slice $\Grpd_{/\T}$ but with a comultiplication
  different from that in Example~\ref{ex:free}: a tree $t$ is comultiplied
  $$
  \Delta(t) = \sum_{\text{subtree covers}} \textstyle{\prod_i s_i} \tensor q
  $$
  by summing over all ways to cover the tree with subtrees $s_i$, disjoint on 
  nodes, then interpreting
  those subtrees as a forest $\prod_i s_i$  (the left-hand tensor factor), and 
  contracting each subtree $s_i$ to a corolla to obtain a quotient tree $q$
  (the right-hand tensor factor).  
  
  Taking cores constitutes a bialgebra homomorphism to the Hopf algebra of
  trees of 
  Calaque--Ebrahimi-Fard--Manchon~\cite{Calaque-EbrahimiFard-Manchon:0806.2238},
  which is of interest since it governs substitution of Butcher series in
  numerical analysis~\cite{Chartier-Hairer-Vilmart:FCM2010}.  Our general result now gives a Fa\`a di Bruno formula in
  the incidence bialgebra, and via the core homomorphism a Fa\`a di Bruno
  sub-bialgebra in the Calaque--Ebrahimi-Fard--Manchon Hopf algebra, which 
  seems not to have been noticed before.

  For general $\RRR$-algebras $A$, i.e.~operads, the constructions and descriptions are the
  same, except that the trees are now $A$-trees, i.e.~trees decorated with 
  $A$-operations on the nodes.  The bialgebra has the same description again,
  but it should be noted that contracting a subtree to a corolla,
  and still obtain an operation to decorate the resulting node with, involves 
  the operad structure of $A$.
\end{blanko}

\begin{blanko}{Feynman graphs.}\label{ex:QFT}
  The 1PI connected Feynman graphs for a given quantum field theory form the
  operations of an operad in groupoids~\cite{Kock:graphs-and-trees}, which we 
  now take as $\RRR$: let $I$
  denote the groupoid of interaction labels (connected graphs without internal
  lines), let $\G$ denote the groupoid of all 1PI connected Feynman graphs with
  residue in $I$, and let $\G\upperstar$ denote the groupoid of all such graphs, but with a
  marked vertex.  The polynomial representing the operad $\RRR$ is
  $$
  I\stackrel s\leftarrow \G\upperstar \stackrel p\to \G \stackrel t \to I
  $$
  where $s$ returns the marked vertex, $p$ forgets the marking, and $t$ returns
  the residue of the graph.  The monad multiplication is given by substitution
  of graphs into vertices.  The resulting bialgebra is the bialgebra version
  \cite{Kock:1411.3098} of the Connes--Kreimer Hopf algebra of
  graphs~\cite{Connes-Kreimer:9912092}.  The connected Green function is the
  standard (bare) combinatorial Green function in quantum field
  theory~\cite{Bergbauer-Kreimer:0506190}, and the Fa\`a di Bruno formula is a
  non-renormalised version of the formula of van
  Suijlekom~\cite{vanSuijlekom:0807}.
\end{blanko}

We now pass to the noncommutative setting: $\SSS$ is now the
monoidal category monad, so that polynomial monads over it are nonsymmetric
operads.  The resulting bialgebras at the vector-space level are now free
noncommutative.

\begin{blanko}{The noncommutative Fa\`a di Bruno bialgebra.}\label{ex:LN}
  Let $\RRR$ be the reduced part of $\SSS$, i.e.~the semimonoid monad.
  The polynomials representing these monads are
    $$\xymatrix{
     1 \ar[d] & \ar[l] \N'_{>0}\drpullback \ar[r]\ar[d] & \N_{>0}\ar[d]\ar[r] & 
     1 \ar[d] \\
     1 & \ar[l] \N'\ar[r] & \N\ar[r] & 1 
  }$$
  where the set of natural numbers $\N$ is a skeleton of the groupoid of finite 
  ordered sets and monotone bijections, and $\N'$ is a skeleton of the groupoid
  of finite pointed ordered sets and basepoint-preserving monotone bijections.
  Hence the fibre of $\N'\to\N$ over $n$ is the linear order $\underline n$.
  
  The
  bar construction is the fat nerve of the category of linearly ordered finite
  sets and monotone surjections.  Since linearly ordered sets have no
  automorphisms, this is equivalent to the strict nerve of any skeleton of the
  category.  The bialgebra in this case 
  is the (dual) Landweber--Novikov bialgebra from algebraic
  topology (see for example~\cite{Morava:Oaxtepec}), which is called the noncommutative
  Fa\`a di Bruno bialgebra by Brouder, Frabetti and 
  Krattenthaler~\cite{Brouder-Frabetti-Krattenthaler:0406117}.
  It is also isomorphic (over $\Q$) to the Dynkin--Fa\`a di Bruno bialgebra,
  introduced in the theory of numerical integration on manifolds by
  Munthe-Kaas~\cite{MuntheKaas:BIT95}; see also
  \cite{EbraihimiFard-Lundervold-Manchon:1402.4761}.
  In the connected Green function and in the Fa\`a di Bruno formula,
  the homotopy sum $\int^k$ is now an ordinary
  sum $\sum_k$ since there are no symmetries present.  The formula 
  $$
  \Delta(A) = \sum_k A^k \tensor A_k
  $$
  is precisely the noncommutative Fa\`a di Bruno 
  formula of \cite{Brouder-Frabetti-Krattenthaler:0406117}, modulo the shift
  in indexing mentioned in \ref{rmk:grading}.
  
  One may consider also a multivariate version of this noncommutative Fa\`a di 
  Bruno bialgebra, proceeding as in \ref{ex:multivar}, but without symmetries.
  The case of two variables is treated in \cite{Brouder-Frabetti-Krattenthaler:0406117}.
\end{blanko}

\begin{blanko}{Braided operads.}
  Take $\SSS$ to be the braided monoidal category monad, which is
  polynomial, represented by the classifying space of the braid groups.
  Taking $\RRR$ to be the commutative monoid monad (denoted $\SSS$ in
  the earlier sections), we obtain a braided version of the Fa\`a di Bruno
  bialgebra, corresponding to the simplicial groupoid $\ourX_\bullet$
  whose codescent object is the category of vines~\cite{Weber-CodescCrIntCat}.
  The braiding is only manifest at the groupoid level though, as taking
  cardinality turns the braided monoidal structure into a commutative monoid.
\end{blanko}

To finish we consider a few silly examples of $\SSS$ just to illustrate 
different instances of the Fa\`a di Bruno formula in certain degenerate 
situations.

\begin{blanko}{Two rather trivial examples.}
  Let $\SSS$ be the identity monad.  A monad cartesian over $\SSS$ is then precisely 
  a small category. The finiteness condition is then for that category to be 
  locally finite.
  
  The $\SSS$-algebra structure is void, so the result of
  the construction is the usual incidence coalgebra of the
  category~\cite{Galvez-Kock-Tonks:1512.07573}.
  The connected Green function $G$ is then simply the sum of all arrows 
  (everything is connected).  We have $G = \sum_k g_k$ where $k$ runs over all
  the objects and $g_k$ denotes the set of arrows with domain $k$.
  The Fa\`a di Bruno formula reads
  $$
  \Delta(G) = \sum_k G^k \tensor g_k .
  $$
  Here $G^k$ is the groupoid of $1$-tuples of arrows with codomain $k$.
  
  In the special case where $\RRR$ is just a monoid (i.e.~a category with only one 
  object), then the finiteness condition amounts to the finite-decomposition 
  property of Cartier--Foata~\cite{Cartier-Foata}, and 
  the Fa\`a di Bruno formula reduces to
  $$
  \Delta(G)= G \tensor G
  $$
  ---the connected Green function is group-like in this case.

  As a slight elaboration on this example, 
  take $\SSS$ to be the pointed set monad.  This will
  yield a pointed comonoid.  The resulting Fa\`a di Bruno formula has only two
  terms.
  
  Let's take the one-object case.  A monad over $\SSS$ is a monoid $M$ together
  with a left module $E$.  One can think of this as an operad with only nullary
  operations $E$ and unary operations $M$.  The generators for the pointed
  coalgebra is the set $\ourX_1=1+E+M$.  We have $\Delta(1) = 1 \tensor 1$, and for
  each $e\in E$ we have $\Delta(e)= 1\tensor e + \sum_{x m = e} x \tensor
  m$ (with $x\in E$ and $m\in M$).  Finally for $a\in M$ we have $\Delta(a) =
  \sum_{nm=a} n \tensor m$.
  
  Inside $\ourX_1$, the connected Green function is given by $G=E+M 
  \subset 1+E+M$.  We have
  $$
  \Delta(G) = 1 \tensor E + G \tensor M  .
  $$
  In a tree interpretation, this says that the
  `cuts' are either empty-followed-by-nullary
  or anything-followed-by-unary.
\end{blanko}

\noindent {\bf Acknowledgements.} We are indebted to Imma G\'alvez and Andy
Tonks for many discussions on the Fa\`a di Bruno formula over the years---we 
owe of course a lot to 
\cite{GalvezCarrillo-Kock-Tonks:1207.6404}--\cite{GKT:ex}%
---and
more specifically for help with Example~\ref{ex:LN}. 
J.K.~was supported by grant number MTM2013-42293-P of Spain.
M.W.~acknowledges the support of the Australian Research Council grant No.~DP130101172, and the Institute of Mathematics of the Czech Academy of Sciences.


\end{document}